\documentclass[11pt]{article}

\usepackage{amsmath,amssymb,amsthm,comment}
\usepackage{tikz}
\usepackage[colorlinks,citecolor=red,urlcolor=blue,linkcolor=blue]{hyperref}
\usepackage{mathrsfs,dsfont}
\usepackage{graphicx}

\newtheorem{nummer}{ }
\newtheorem{theorem}[nummer]{\bf Theorem}
\newtheorem{proposition}[nummer]{\bf Proposition}
\newtheorem{lemma}[nummer]{\bf Lemma}
\newtheorem{corollary}[nummer]{\bf Corollary}
\newtheorem{definition}[nummer]{\bf Definition}

\begin{document}
\begin{center}
{\LARGE\bf Reversion Porisms in Conics}

\bigskip
\textcolor{black}{%
{\small Lorenz Halbeisen}\\[1.2ex] 
{\scriptsize Department of Mathematics, ETH Zentrum,
R\"amistrasse\;101, 8092 Z\"urich, Switzerland\\ lorenz.halbeisen@math.ethz.ch}}\\[1.8ex]

\textcolor{black}{%
{\small Norbert Hungerb\"uhler}\\[1.2ex] 
{\scriptsize Department of Mathematics, ETH Zentrum,
R\"amistrasse\;101, 8092 Z\"urich, Switzerland\\ norbert.hungerbuehler@math.ethz.ch}}\\[1.8ex]

\textcolor{black}{%
{\small Marco  Schiltknecht}\\[1.2ex] 
{\scriptsize Department of Mathematics, ETH Zentrum,
R\"amistrasse\;101, 8092 Z\"urich, Switzerland\\ marcosc@student.ethz.ch}}
\end{center}
\small{\it key-words\/}: porisms, butterfly theorems, reversion map


\begin{abstract}\noindent
We give a projective proof of the butterfly porism for cyclic quadrilaterals and present a general reversion porism
for polygons with an arbitrary number of vertices on a conic. We also investigate projective properties of the porisms.
\end{abstract}

\section{Introduction}
The theorems of Pappus and Pascal and the Scissors Theorem can be formulated as porisms in the projective plane:
\begin{theorem}[Pappus]\label{thm-pappus}
Let $A_1,A_2,\ldots,A_6$ be a Pappus hexagon on the lines $\ell_1,\ell_2$ with intersection points
$P_1,P_2,P_3$ on the Pappus line $\ell$. Then there exists a Pappus hexagon $A'_1,A'_2,\ldots,A'_6$
on $\ell_1,\ell_2$ with the same intersection points $P_1,P_2,P_3$ for any point $A'_1$ on
$\ell_1$ 
(see Figure~\ref{fig-pappus}).
\end{theorem}
The cases when $A_1'$ is the intersection of $\ell_1$ with $\ell$ or $\ell_2$ are considered as degenerate situations.
\begin{proof}[Proof of Theorem\;\ref{thm-pappus}]
By the Theorem of Pappus, applied to the hexagon $A_1,\ldots,A_6$, the points $P_1,P_2,P_3$ are collinear.
Then the Braikenridge-Maclaurin Theorem for degenerate conics (see, e.g.,~\cite[p.~76]{coxeter}) applied to the points $A'_1,A'_2,\ldots,A'_5$ and the points $P_1,P_2,P_3$
implies that $A'_6$ lies on $\ell_2$.
\end{proof}
\begin{figure}[ht!]
\begin{center}
\includegraphics{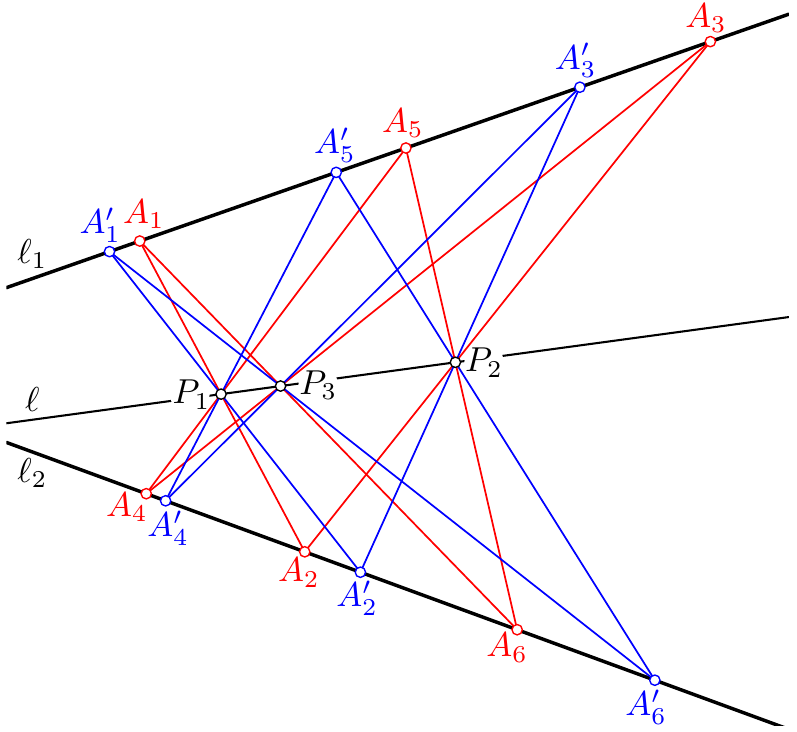}
\caption{Pappus hexagons $A_1,A_2,\ldots,A_6$ and $A'_1,A'_2,\ldots,A'_6$.}\label{fig-pappus}
\end{center}
\end{figure}

\begin{theorem}[Pascal]\label{thm-pascal}
Let $A_1,A_2,\ldots,A_6$ be a Pascal hexagon on the conic $C$ with intersection points
$P_1,P_2,P_3$ on the Pascal line $\ell$. Then there exists a Pascal hexagon $A'_1,A'_2,\ldots,A'_6$
on $C$ with the same intersection points $P_1,P_2,P_3$ for any point $A'_1$ on
$C$ 
(see Figure~\ref{fig-pascal}).
\end{theorem}
\begin{figure}[ht!]
\begin{center}
\includegraphics{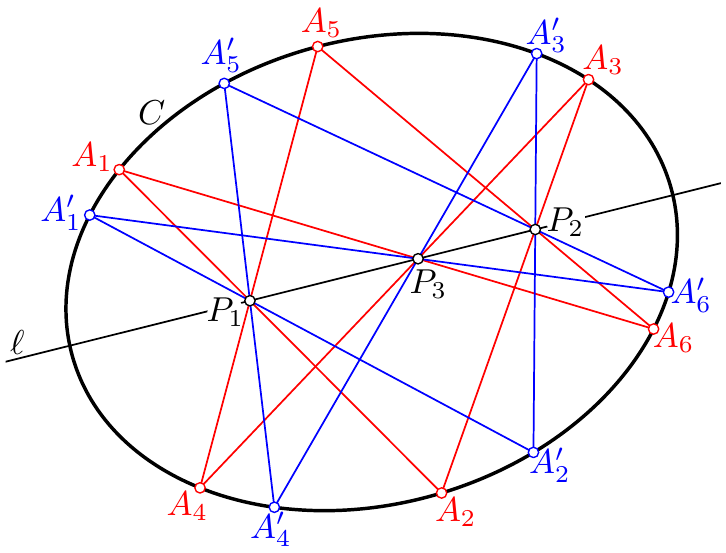}
\caption{Pascal hexagons $A_1,A_2,\ldots,A_6$ and $A'_1,A'_2,\ldots,A'_6$.}\label{fig-pascal}
\end{center}
\end{figure}

\noindent Again, we consider the case when $A_1'$ is on $\ell$ as a degenerate situation.
\begin{proof}[Proof of Theorem\;\ref{thm-pascal}]
By the Theorem of Pascal, applied to the hexagon $A_1,\ldots,A_6$, the points $P_1,P_2,P_3$ are collinear.
Then the Braikenridge-Maclaurin Theorem for nondegenerate conics applied to the points $A'_1,A'_2,\ldots,A'_5$ and the points $P_1,P_2,P_3$
implies that $A'_6$ lies on $C$.
\end{proof}

\begin{theorem}[Scissors Theorem]\label{thm-scissors}
Let $A_1,A_2,A_3,A_4$ be a Scissors quadrilateral on the lines $\ell_1,\ell_2$ with intersection points
$P_1,P_2,P_3,P_4$ on a  line $\ell$. Then there exists a Scissors quadrilateral $A'_1,A'_2,A'_3,A'_4$
on $\ell_1,\ell_2$ with the same intersection points $P_1,P_2,P_3,P_4$ for any point $A'_1$ on
$\ell_1$ 
(see Figure~\ref{fig-scissors}).
\end{theorem}
\begin{figure}[ht!]
\begin{center}
\includegraphics{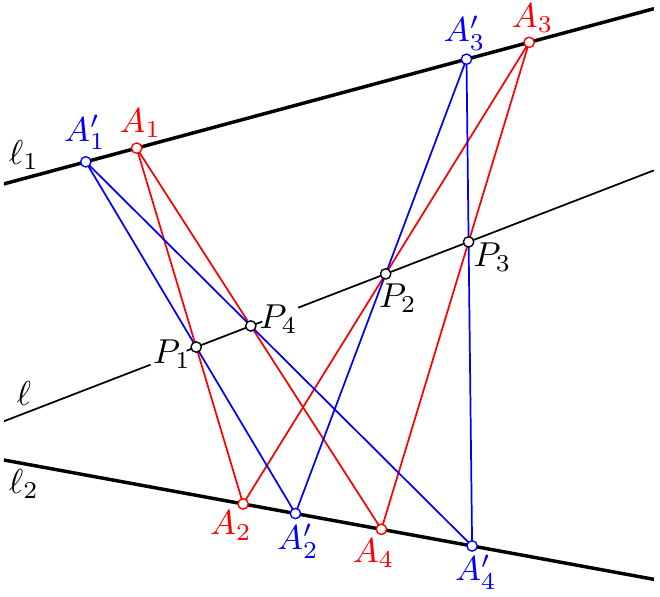}
\caption{Scissors quadrilaterals $A_1,A_2,\ldots,A_4$ and $A'_1,A'_2,\ldots,A'_4$.}\label{fig-scissors}
\end{center}
\end{figure}

\noindent Also here, the cases when $A_1'$ is the intersection of $\ell_1$ 
with $\ell$ or $\ell_2$ are included in the theorem as  degenerate situations.
\begin{proof}[Proof of Theorem\;\ref{thm-scissors}]
In the real affine plane obtained by removing $\ell$ from the projective plane
the lines $A_iA_{i+1}$ and $A'_iA'_{i+1}$ (with  cyclically read indices) are parallel.
Then the quadrilateral $A'_1,A'_2,A'_3,A'_4$ is
obtained from $A_1,A_2,A_3,A_4$  by a translation (if $\ell_1,\ell_2$ and $\ell$ are concurrent)
or by a homothetic transformation (if $\ell_1,\ell_2$ and $\ell$ are not concurrent).
\end{proof}

Formally, one obtains the Pascal Porism~\ref{thm-pascal} from the Pappus Porism~\ref{thm-pappus}
by replacing the degenerate conic $\ell_1\cup\ell_2$ by the nondegenerate conic $C$.
Surprisingly, one obtains another porism by the same process when we replace the 
degenerate conic $\ell_1\cup\ell_2$ in the Scissors Porism~\ref{thm-scissors} by a nondegenerate
conic $C$:
\begin{theorem}[Butterfly Theorem, Jones \cite{jones}, Kocik \cite{Kocik}]\label{thm-kocik}
Let $A_1,A_2,A_3,A_4$ be a quadrilateral on the conic $C$ with intersection points
$P_1,P_2,P_3,P_4$ on a  line $\ell$. Then there exists a quadrilateral $A'_1,A'_2,A'_3,A'_4$
on $C$ with the same intersection points $P_1,P_2,P_3,P_4$ for any point $A'_1$ on
$C$ 
(see Figure~\ref{fig-kocik}).
\end{theorem}
\begin{figure}[ht!]
\begin{center}
\includegraphics{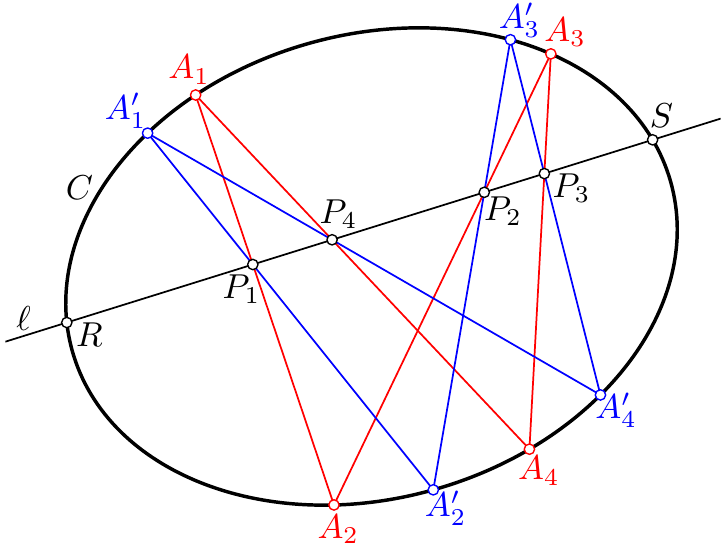}
\caption{Butterfly quadrilaterals $A_1,A_2,\ldots,A_4$ and $A'_1,A'_2,\ldots,A'_4$.}\label{fig-kocik}
\end{center}
\end{figure}
The theorem of Jones and Kocik generalizes the well known classical Butterfly 
Theorem (see, e.g.,~\cite{Klamkin}, \cite{Ana}, \cite{Volenec},
and the many variants in~\cite{ctk}).
Kocik proved the theorem in the setting of the complex plane using M\"obius transformations, thus excluding the case that
one of the points $P_i$ lies on the ideal line. He also proved a corresponding theorem for
polygons $A_1,A_2,\ldots,A_{2n}$, $n\ge 2$, if the points $P_1,P_2,\ldots,P_{2n}$ are collinear (see~\cite{Kocik}).
Izmestiev quantified in~\cite{Izmestiev} Kocik's Theorem using cross-ratios by proving, that
\begin{equation}\label{eq-DV}
\operatorname{cr}(R,S;P_1,P_4) =\operatorname{cr}(R,S;P_2,P_3),
\end{equation}
is a necessary and sufficient condition for the closing of the quadrilaterals.
Here $\ell$ intersects $C$ in the points $R,S$ (see Figure~\ref{thm-kocik}). 
(The same relation holds for the Scissors Porism~\ref{thm-scissors} when $R$ and $S$ are the intersections of $\ell$ with
$\ell_1$ and $\ell_2$, respectively.)
Izmestiev gave a similar characterisation if $\ell$ is tangent to $C$ or does not intersect $C$ (see~\cite{Izmestiev}).
He also proved that the closing condition is satisfied for a polygon $A_1,A_2,\ldots,A_n$, $n\ge 5$, in a circle $C$
if the points $P_1,P_2,\ldots,P_n$ form a right-angled polygon when the interior of $C$
is viewed as the Cayley-Klein model of the hyperbolic plane. This is a sufficient, but not
a necessary closing condition. In particular, it only applies to points $P_i$ inside $C$. Also Izmestiev used
M\"obius transformations  and was therefore subject to the same limitations as Kocik.

In Section~\ref{sec-kocik} we will prove the Butterfly Porism~\ref{thm-kocik} in a purely projective manner
using projective maps instead of M\"obius transformations, thus closing the gap in the previous proofs. 
In Section~\ref{sec-general} we will give a necessary and sufficient closing condition for the points  $P_1,\ldots,P_n$, $n\ge 3$:
It will turn out that the condition is a consequence of Pascal's Theorem. In particular,
given an $n$-gon $A_1,A_2,\ldots,A_n$ inscribed in a conic $C$ with points $P_1,\ldots,P_{n-2}$ such that $P_i$ lies on 
the line $A_iA_{i+1}$, there are unique points $P_{n-1}$ on  $A_{n-1}A_n$ and $P_n$ on  $A_nA_1$
such that the closing condition is satisfied. I.e., there is an $n$-gon $A_1'A_2'\ldots A_n'$ on $C$ with
sides running successively through the points $P_i$ for any starting point $A_1'$ on $C$. 
The points $P_{n-1}$ and $P_n$ can
easily be constructed by ruler alone.

\section{The Butterfly Porism}\label{sec-kocik}
Let $C$ be a nondegenerate conic in the real projective plane $\mathbb R\!\operatorname{P}^2$.
For a point $P\notin C$,
the {\em reversion map\/} $\varphi_P:C\to C, X\mapsto \varphi_P(X)$, is defined by the requirement that
$X,P,\varphi_P(X)$ are collinear, and that $X\neq\varphi_P(X)$ unless $XP$ is a tangent of $C$ (see Figure~\ref{fig-reversion}). 
If convenient, we may always assume without loss of generality that the conic $C$ is given by
$$C: \langle X,X\rangle = 0$$
in projective coordinates $X=(x_1,x_2,x_3)^\top$, where $\langle X,Y\rangle=x_1y_1+x_2y_2-x_3y_3$ denotes
the Minkowski product. In the Euclidean plane $\{(x_1,x_2,1)\mid (x_1,x_2)\in\mathbb R^2\}$,
embedded in the projective plane $\mathbb R\!\operatorname{P}^2$,  $C$ is the unit circle. 
\begin{figure}[ht!]
\begin{center}
\includegraphics{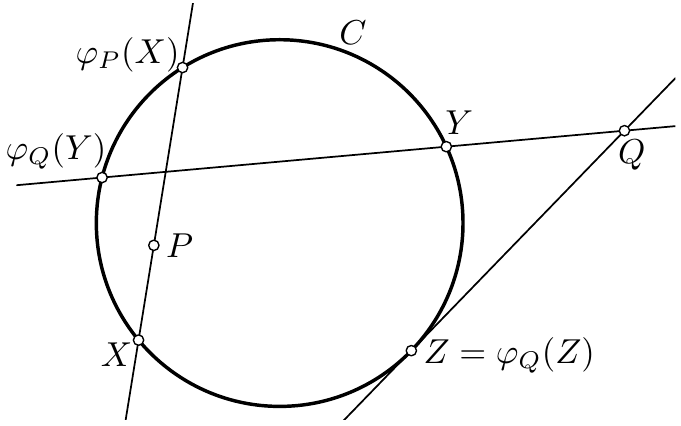}
\caption{The reversion map on $C$.}\label{fig-reversion}
\end{center}
\end{figure}

Observe that $\varphi_P$ has a unique continuation
as a projective map $\mathbb R\!\operatorname{P}^2\to\mathbb R\!\operatorname{P}^2$ 
which we also denote by $\varphi_P$ and
which is given by
$$
X\mapsto\varphi_P(X)=M_PX
$$
where 
$$
M_P=\begin{pmatrix}
-p_1^2+p_2^2-p_3^2 & -2p_1p_2 & 2 p_1p_3\\
-2p_1p_3 &p_1^2-p_2^2-p_3^2 & 2p_2p_3\\
-2p_1p_3 &-2p_2p_3& p_1^2+p_2^2+p_3^2
\end{pmatrix}.
$$
Note that the polar line of $P$ is a fixed point line of $\varphi_P$ and
the  bundle of lines through $P$ are fixed lines of $\varphi_P$. Moreover, $\varphi_P$ is an involution.
\begin{lemma}\label{lem-kocik}
If $U,V,W\in \mathbb R\!\operatorname{P}^2\setminus C$ are collinear, then there is a unique
point $X \in \mathbb R\!\operatorname{P}^2\setminus C$ on the same line as $U,V,W$ such that
$\varphi_X\circ\varphi_W\circ\varphi_V\circ\varphi_U=\operatorname{id}$, i.e.,
the identity.
\end{lemma}
\begin{proof}
$W$ is given by $W=aU+bV$ for some $a,b\in\mathbb R$. Observe that
$$
\langle U,U\rangle\neq 0,\quad
\langle V,V\rangle\neq 0,\quad
\langle W,W\rangle
\neq 0,
$$
since $U,V,W\notin C$. Set 
$X:=(2a\langle U,V\rangle + b\langle V,V\rangle)U - a\langle U,U\rangle V$. Then
a short calculation shows that
$$
\langle X,X\rangle=\langle U,U\rangle \langle V,V\rangle \langle W,W\rangle \neq 0
$$
and hence $X\notin C$. It is then elementary to
check that $M_X=M_WM_VM_U$.
\end{proof}
The Butterfly Porism~\ref{thm-kocik}  follows immediately from Lemma~\ref{lem-kocik} with $P_1=U, P_2=V, P_3=W$:
Indeed, Lemma~\ref{lem-kocik} guarantees the existence of a point $X$ on the line $\ell$ such that
$A_2=\varphi_{P_1}(A_1),A_3=\varphi_{P_2}(A_2),A_4=\varphi_{P_3}(A_3)$, and $A_1=\varphi_{X}(A_4)$.
Thus, $X=P_4$, and because $\varphi_{P_4}\circ\ldots\circ\varphi_{P_1}=\operatorname{id}$,
the path closes for any starting point $A'_1$ on $C$.

Is it possible that a closing theorem for quadrilaterals in a conic $C$ holds
if the points $P_1,\ldots,P_4$ are {\em not\/} collinear? The answer is no:
\begin{theorem}
Let $A_i,B_i,C_i,D_i$ for $i=1,2,3$ be three quadrilaterals with vertices on a conic $C$ with points 
$P_1$ on  $A_i B_i$,
$P_2$ on $B_i C_i$,
$P_3$ on $C_i D_i$,
$P_4$ on $D_i A_i$, for $i=1,2,3$. Assume that the four points $P_j$ are not on $C$.
Then, $P_1,P_2,P_3,P_4$ are collinear.
\end{theorem}
\begin{proof}
Consider the following four hexagons
\begin{eqnarray}
&& A_1 B_1 D_2 A_2 B_2 D_1, \label{h1}\\
&& A_1 B_1 D_3 A_3 B_3 D_1, \label{h2}\\
&& C_1 B_1 D_2 C_3 B_2 D_1, \label{h3}\\
&& C_1 B_1 D_3 C_3 B_3 D_1. \label{h4}
\end{eqnarray}
Let $X$ be the intersection of the lines $B_1D_2$ and $B_2D_1$ and
$Y$ the intersection of the lines $B_1D_3$ and $B_3D_1$.
Then, by Pascal's Theorem applied to the four hexagons above, we have:
\begin{eqnarray*}
&&\text{by~(\ref{h1}): the points $P_1,P_4,X$ are collinear} \\
&&\text{by~(\ref{h2}): the points $P_1,P_4,Y$ are collinear} \\
&&\text{by~(\ref{h3}): the points $P_2,P_3,X$ are collinear} \\
&&\text{by~(\ref{h4}): the points $P_2,P_3,Y$ are collinear}
\end{eqnarray*}
Hence, $P_1,P_2,P_3,P_4$ are collinear.
\end{proof}

\section{General reversion porisms in conics}\label{sec-general}
Motivated by the previous section we define:
\begin{definition}
Let $C$ be a conic in the projective plane $\mathbb R\!\operatorname{P}^2$.
The points $P_1,\ldots, P_n$ in $\mathbb R\!\operatorname{P}^2\setminus C$ are said to
satisfy the closing property (in this order) with respect to $C$, 
if $\varphi_{P_n}\circ\ldots\circ\varphi_{P_1}=\operatorname{id}$, i.e., the identity.
\end{definition}
Then we have the following:
\begin{lemma}
Let $C$ be a nondegenerate conic in the real projective plane $\mathbb R\!\operatorname{P}^2$ and
$P_1,\ldots, P_n$  be points in $\mathbb R\!\operatorname{P}^2\setminus C$.
Assume that there are three different $n$-gons $A^{j}_1A^{j}_2\ldots A^{j}_n$, $j=1,2,3,$ 
inscribed in $C$ such that $P_i$ lies on $A_iA_{i+1}$ for $i=1,2,\ldots,n$ (with cyclically read indices).
Then $P_1,\ldots, P_n$ have the closing property with respect to $C$. In particluar
there is a closed $n$-gon starting in any point $A_1\in C$ whose sides run successively through the points $P_i$. 
\end{lemma}
\begin{proof}
$A_1^1,A_1^2$ and $A_1^3$ are three different fixed points of the map 
$\varphi_{P_n}\circ\ldots\circ\varphi_{P_1}$.
Thus, the claim follows directly from the fact that the  group of 
projective maps which keep the set $C$ fixed acts sharply 3-transitively on $C$.
\end{proof}
Before we can turn to the main theorem, we need to state the following:
\begin{proposition}\label{prop-ell}
Let $C$ be a nondegenerate conic and  $U\neq V$ be points in $\mathbb R\!\operatorname{P}^2\setminus C$.
Then 
$\varphi=\varphi_V\circ\varphi_U$ 
has the line $\ell=UV$ as a fixed line and its pole $P$ with respect to $C$ as a fixed point. Moreover we have:
\begin{itemize}
\item[(a)] If $\ell$ intersects $C$ in two points $R,S$ then $R,S$ are fixed points of $\varphi$
and the tangents in $R,S$ are fixed lines. Besides $P$ and $\ell$ there are no other fixed points and fixed lines.
\item[(b)] If $\ell$ is tangent to  $C$ or if $\ell$ misses $C$, then $P$ is the only fixed point and $\ell$ the only fixed line.
\end{itemize}
In particular, the line $\ell$ on which $U$ and $V$ sit is determined by the map $\varphi$.
\end{proposition}
\begin{proof}
The fixed points of $\varphi$ are the real eigenvectors of $M=M_VM_U$ and the fixed lines
are the real eigenvectors of $M^{-\top}$ (the inverse transposed of $M$). In case (a) it is 
geometrically clear, that $R,S,P$ are fixed points and that $\ell$ and the tangents in $R,S$ are
fixed lines. Since we have at most three real eigenvectors, there are no other fixed points or lines.
In case (b) a short calculation shows, that $P$ is a triple eigenvector of $M$, and $\ell$ a triple
eigenvector of $M^{-\top}$ if $\ell$ is tangent to $C$. If $\ell$ misses $C$, there is only one real eigenvalue of $M$, namely $P$, 
and only one real eigenvalue of $M^{-\top}$, namely $\ell$.
\end{proof}
Now we are ready for the main theorem.
\begin{theorem}\label{thm-main}
Let $C$ be a nondegenerate conic and $P_1, \ldots,P_{n-2}$ be given points in 
$\mathbb R\!\operatorname{P}^2\setminus C$.
Then the following is true:
\begin{itemize}
\item If $P_1, \ldots,P_{n-2}$ have the closing property with respect to $C$, then
$P_1, \ldots,P_{n-2},$ $P_{n-1},P_n$ have the closing property 
if and only if $P_{n-1}=P_n$ is an arbitrary point in $\mathbb R\!\operatorname{P}^2\setminus C$.
\item  If $P_1, \ldots,P_{n-2}$ do not have the closing property with respect to $C$, \textcolor{black}{then} either
$\varphi=\varphi_{P_{n-2}}\circ\ldots\circ\varphi_{P_1}$ has a unique fixed line $\ell$ which intersects $C$ in two points, or
$\varphi$ has only one fixed line $\ell$ at all. 
For an arbitrary point $P_{n-1}\in \ell \setminus C$
there is a unique point $P_{n}\in \ell\setminus C$ such that $P_1, \ldots,P_{n-2},P_{n-1},P_n$ have the closing property.
No other choice for $P_{n-1}$ and $P_n$ is possible.
\end{itemize}
\end{theorem}
\begin{proof}
The first case is trivial, since $\varphi_{P_n}\circ\varphi_{P_{n-1}}\circ \ldots\circ\ \varphi_{P_{1}}=\varphi_{P_{n}}\circ\varphi_{P_{n-1}}=
\operatorname{id}$ implies that $\varphi_{P_{n-1}}=\varphi_{P_n}$ and hence $P_{n-1}=P_n$.

In the second case we assume that $\varphi=\varphi_{P_{n-2}}\circ \ldots\circ\varphi_{P_1}\neq\operatorname{id}$.
We start by showing that $P_{n-1}$ and $P_n$ exist as specified in the theorem.
In this case, $\varphi$ has at most two fixed points on $C$. So, let us first choose an arbitrary point $A_1\neq\varphi(A_1)$.
Then, we can choose two different points $A_1',A_1''\notin\{ \varphi^{-1}(A_1),\varphi(A_1)\}$ such that $A_1'\neq \varphi(A_1')$
and $A_1''\neq \varphi(A_1'')$.  This defines the polygonal chains 
\begin{eqnarray*}
&&A_1,A_2=\varphi_{P_1}(A_1),A_3=\varphi_{P_2}(A_2),\ldots,A_{n-1}=\varphi_{P_{n-2}}(A_{n-2})=\varphi(A_1),\\
&&A'_1,A'_2=\varphi_{P_1}(A'_1),A'_3=\varphi_{P_2}(A'_2),\ldots,A'_{n-1}=\varphi_{P_{n-2}}(A'_{n-2})=\varphi(A'_1),\\
&&A''_1,A''_2=\varphi_{P_1}(A''_1),A''_3=\varphi_{P_2}(A''_2),\ldots,A''_{n-1}=\varphi_{P_{n-2}}(A''_{n-2})=\varphi(A''_1).
\end{eqnarray*}
Then the intersection $X$ of the lines $A_1A'_{n-1}$ with the line $A_1'A_{n-1}$
and the intersection $Y$ of the lines $A_1A''_{n-1}$ with the line $A_1''A_{n-1}$ are different and 
define a line $\ell$ (see Figure~\ref{fig-grund}). Choose a point $P_{n-1}$ on $\ell$ such that the line $A_{n-1}P_{n-1}$
intersects $C$ in a point $A_n$. The line $A_nA_1$ then intersects $\ell$ in a point $P_n$. Hence, 
$A_n=\varphi_{P_{n-1}}(A_{n-1})$ and $A_1=\varphi_{P_{n}}(A_{n})$.
\begin{figure}[ht!]
\begin{center}
\includegraphics{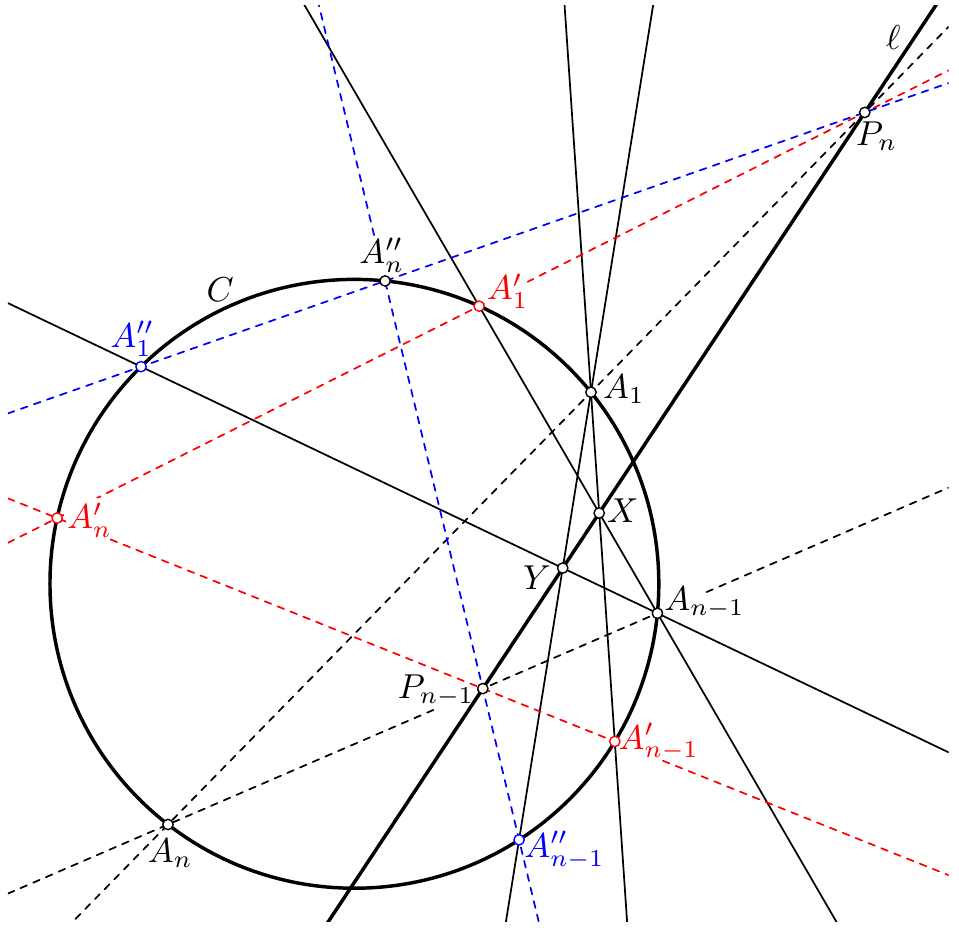}
\caption{Proof of Theorem~\ref{thm-main}}\label{fig-grund}
\end{center}
\end{figure}
Now we consider the intersection  $A'_n$ of the lines $A'_{n-1}P_{n-1}$ and $A'_1P_n$ and
the intersection $A''_n$ of the lines $A''_{n-1}P_{n-1}$ and $A''_1P_n$. By the Braikenridge-Maclaurin Theorem
applied to the hexagon $H_1=A_1 A'_{n-1} A'_n A'_1 A_{n-1} A_n$ it follows that $A'_n\in C$.
Similarly, by considering the hexagon $H_2=A_1 A''_{n-1} A''_n A''_1 A_{n-1} A_n$ it follows that $A''_n\in C$.
It follows that the map $\varphi_{P_n}\circ \ldots \circ \varphi_{P_1}$ has the fixed points
$A_1,A_1'$ and $A_1''$ on $C$ and is hence the identity. In particular, we see that
$$
\varphi=\varphi_{P_{n-2}}\circ \ldots \circ \varphi_{P_{1}} = \varphi_{P_{n-1}}\circ  \varphi_{P_{n}}.
$$
Hence, by Proposition~\ref{prop-ell}, the line $\ell$ on which $P_{n-1}$ and $P_n$ sit
is determined by the points $P_1,\ldots,P_{n-2}$, and clearly, $P_n$ is determined
as soon as $P_{n-1}$ is chosen on $\ell\setminus C$.

Conversely, if we assume that $\varphi_{P_{n}}\circ \ldots \circ \varphi_{P_{1}} =\operatorname{id}$,
there are three $n$-gons $A_1,\ldots, A_n$, $A'_1,\ldots, A'_n$ and $A''_1,\ldots, A''_n$ as in Figure~\ref{fig-grund}.
Then, by the Pascal Theorem applied to the hexagons $H_1$ and $H_2$ mentioned above, it follows
that $P_{n-1}$ and $P_n$ must lie on the common Pascal line $\ell$ of the two hexagons.
The lines $A_1A'_{n-1}$ and $A_1'A_{n-1}$ determine the point $X$ on $\ell$, and
the lines $A_1A''_{n-1}$ and $A_1''A_{n-1}$ determine the point $Y\neq X$ on $\ell$. Thus, $\ell$
is determined by $P_1,\ldots,P_{n-2}$ by the construction above. Hence $P_{n-1}$ must be chosen on $\ell$
and once $P_{n-1}$ is fixed, the location of $P_n$ on $\ell$ follows.
\end{proof}
Notice that the construction of the points $P_{n-1}$ and $P_n$ can be carried out with ruler alone as
shown in the proof.

For $n=3$ we obtain the nice porism in Figure~\ref{fig-3}. In this case,  $P_1$ is given, and
$\ell$ is the polar line of $P_1$ with respect to $C$.  $P_2$ is chosen freely on $\ell\setminus C$
and then $P_3$ is the conjugate of $P_2$ on $\ell$ with respect to $C$. This can be seen when
the red or blue triangle in $C$ collapses to the line $P_1P_2$ or $P_1P_3$.
\begin{figure}
\begin{center}
\includegraphics{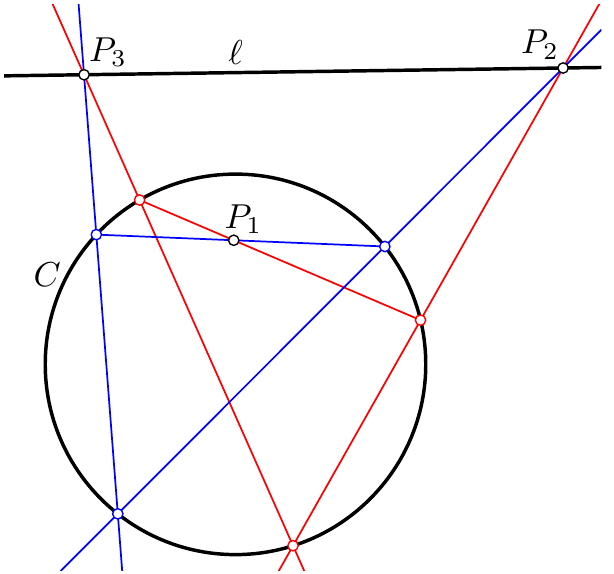}
\caption{A porism for triangles: $\ell$ is the pole of $P_1$. $P_2$ and $P_3$ on $\ell$ are conjugate points with respect to $C$.}\label{fig-3}
\end{center}
\end{figure}
For $n=4$ Theorem~\ref{thm-main} coincides with Theorem~\ref{thm-kocik}.

We close this section with an example of a porism for heptagons in Figure~\ref{fig-7}.
\begin{figure}[ht!]
\begin{center}\vspace*{-5mm}
\includegraphics{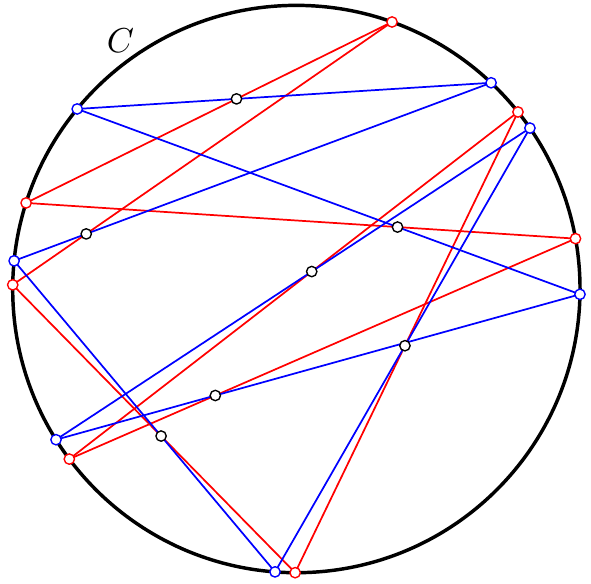}
\caption{A porism for heptagons: The red or blue heptagon can start in any point of $C$.}\label{fig-7}
\end{center}
\end{figure}

\section{Projective aspects of the Butterfly Porism}
The closing condition~(\ref{eq-DV}) allows a nice conclusion:
\begin{corollary}\label{cor-conjugate}
If the points $P_1,\ldots,P_4$ on the line $\ell$ have the closing property with respect
to a conic $C$ and if $\ell$ intersects $C$ in two points 
then the conjugate points $Q_1,\ldots,Q_4$ 
have also the closing property with respect to $C$.
\end{corollary}
\begin{proof}
The cross ratio of four points equals  the cross ratio of the corresponding polar lines, i.e.,
we have
\begin{align*}
&cr(R,S;Q_1,Q_4)=cr(r,s;p_1,p_4)=cr(R,S;P_1,P_4)=\\
&cr(R,S;P_2,P_3)=cr(r,s;p_2,p_3)=cr(R,S;Q_2,Q_3).
\end{align*}
See Figure~\ref{fig-conjugate}. 

\begin{figure}[ht!]
\begin{center}
\includegraphics{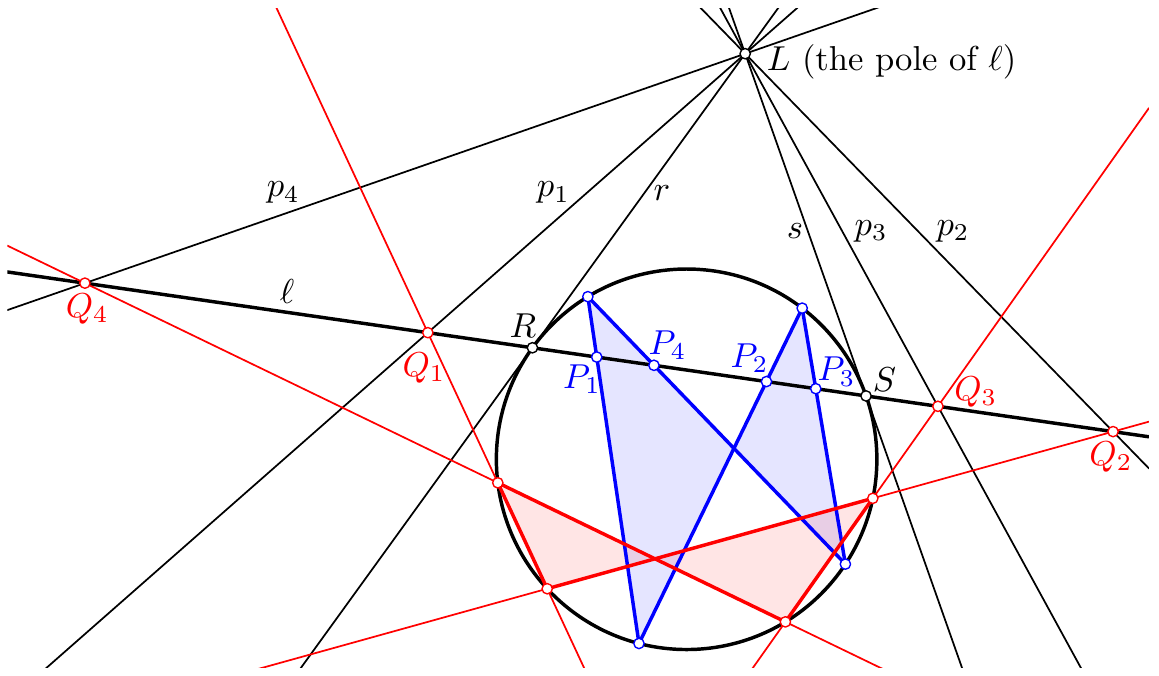}
\caption{The closing property of the conjugate points.}\label{fig-conjugate}
\end{center}
\end{figure}
\end{proof}

Recall that Kocik~\cite{Kocik} proved the Butterfly Porism~\ref{thm-kocik} also for an 
arbitrary even number of collinear points $P_1,P_2,\ldots,P_{2n}$, $n\ge 1$. (It follows
easily from Lemma~\ref{lem-kocik} that the theorem cannot hold for an odd number of points.)
The previous corollary also holds in this case: 
\begin{corollary}
If the points $P_1,\ldots,P_{2n}$ on the line $\ell$ have the closing property with respect
to a conic $C$ and if $\ell$ intersects $C$ in two points 
then the conjugate points $Q_1,\ldots,Q_{2n}$ 
have also the closing property with respect to $C$.
\end{corollary}
\begin{proof}
For $n=1$ the assertion is trivial and the case $n=2$ is Corollary~\ref{cor-conjugate}. We proceed
by induction and suppose that we have proven the claim for some $n\ge 2$.
Let us assume that the points $P_{1},P_2,\ldots,P_{2(n+1)}$ have the closing property
with respect to $C$, i.e.,
$$\varphi_{P_{2(n+1)}}\circ \varphi_{P_{2n+1}}\circ\ldots\circ\varphi_{P_{1}}=\operatorname{id}.$$
By Lemma~\ref{lem-kocik} we have that 
$$
\varphi_{P_{2(n+1)}}\circ\varphi_{P_{2n+1}}\circ \varphi_{P_{2n}}=\varphi_{P}
$$
for a point $P$ on the line $\ell$, and by Corollary~\ref{cor-conjugate}
\begin{equation}\label{eq-i}
\varphi_{Q_{2(n+1)}}\circ\varphi_{Q_{2n+1}}\circ \varphi_{Q_{2n}}=\varphi_{Q}
\end{equation}
for the conjugate point $Q$ of $P$. Thus, we have
$$
\varphi_P\circ\varphi_{P_{2n-1}}\circ\ldots\circ\varphi_{P_1}=\operatorname{id}
$$
and by the induction hypothesis
\begin{equation}\label{eq-ii}
\varphi_Q\circ\varphi_{Q_{2n-1}}\circ\ldots\circ\varphi_{Q_1}=\operatorname{id}.
\end{equation}
The claim follows when we replace $\varphi_Q$ in~(\ref{eq-ii}) by~(\ref{eq-i}).
\end{proof}

If the line $\ell$ is tangent to $C$, then the conjugate points $Q_i$ coincide with the point of contact.
However, we can extend the above result to the case when $\ell$ does not meet $C$:
\begin{theorem}
If the points $P_1,\ldots,P_{2n}$ on the line $\ell$ have the closing property with respect
to a conic $C$ and if $\ell$ does not meet $C$,  
then the conjugate points $Q_1,\ldots,Q_{2n}$ 
have also the closing property with respect to $C$.
\end{theorem}
\begin{proof}
By applying a projective map we may assume that $C$ is the conic given by $\langle X,X\rangle=0$
and $\ell$ the line given by $\langle X,L\rangle=0$ with $L=(0,0,1)$. 
The projective map $\psi:(x_1,x_2,x_3)\mapsto (-x_2,x_1,x_3)$ maps $C$ to $C$, and
points on $\ell$ to the conjugate points with respect to $C$. Thus, every closed
polygon on $C$ with sides running successively through the points $P_1,\ldots,P_{2n}$
is mapped by $\psi$ to a closed polygon on $C$ with sides running successively through the conjugate points $Q_1,\ldots,Q_{2n}$.
\end{proof}
Izmestiev noted, that since the cross ratio is invariant under projective transformations,
the closing condition~(\ref{eq-DV}) holds for an arbitrary non-degenerate conic $C$.
I.e., if $P_1,\ldots,P_4$ on a line $\ell$ have the closing property with respect to a conic $C$
they have also the closing property with respect to any other conic $D$ which intersects
$\ell$ in the same points $R$ and $S$ as $C$. The reasoning is as follows:
There exists a projective map $\psi$ which maps $D$ to $C$ and which has the fixed
points $R$ and $S$. Let $Q_i=\psi(P_i)$ for $i=1,\ldots,4$. Then the points $Q_i$ have the
closing property with respect to $C$ iff $cr(R,S;Q_1,Q_4)=cr(R,S;Q_2,Q_3)$
which is equivalent to the condition $cr(R,S;P_1,P_4)=cr(R,S;P_2,P_3)$.
If it is satisfied, then any closed quadrilateral  on $C$ with sides running successively through
the points $Q_1,\ldots,Q_4$ is mapped by $\psi^{-1}$ to a closed quadrilateral
on $D$ with sides running successively through the points $P_1,\ldots,P_4$.
Notice however, that in general there is no projective map which maps a 
closed polygon on $C$ with sides running through the points $P_1,\ldots,P_4$
to a closed polygon on $D$ with sides running through the points $P_1,\ldots,P_4$.
This is due to the fact that the points  $P_1,\ldots,P_4$ may be inner points of
$C$ but exterior points of $D$.

Even more generally, one can replace $R$ by a point $R'$ on $\ell$ and then
determine the unique point $S'$ such that  $cr(R',S';P_1,P_4)=cr(R',S';P_2,P_3)$.
Then the $P_1,\ldots,P_4$ have the closing property with respect to any conic
$D$ running through $R'$ and $S'$.

We can generalise the observation above from four to an arbitrary even number of points:
\begin{proposition}
If the points $P_1,\ldots,P_{2n}$ on the line $\ell$ have the closing property with respect
to a conic $C$ and if $\ell$  intersects $C$  in two points $R$ and $S$, then
$P_1,\ldots,P_{2n}$ have the closing property with respect to any other
conic $D$ through the points $R$ and $S$.
\end{proposition}
\begin{proof}
For $n=1$ and $n=2$ there is nothing more to prove. So we can proceed by
induction and assume that we have proven the claim for some $n\ge 2$.
Suppose that the points $P_{2(n+1)},P_{2n+1},\ldots,P_1$ on the line $\ell$ have the closing
property with respect to a conic $C$. I.e., we have
$$
\varphi_{P_{2(n+1)}}^C\circ \varphi_{P_{2n+1}}^C\circ\ldots\circ \varphi_{P_1}^C=\operatorname{id}
$$
where $\varphi_{P_i}^C$ means the reversion map with respect to $C$.
Then, by the result for $n=2$ we have
$$
\varphi_{P_{2(n+1)}}^C\circ \varphi_{P_{2n+1}}^C\circ \varphi_{P_{2n}}^C=\varphi_{P}^C
$$
for some $P$ on $\ell$ and also
\begin{equation}\label{eq-I}
\varphi_{P_{2(n+1)}}^D\circ \varphi_{P_{2n+1}}^D\circ \varphi_{P_{2n}}^D=\varphi_{P}^D
\end{equation}
for another conic $D$ through the points $R$ and $S$. So, we have
$$
\varphi_{P}^C\circ \varphi_{P_{2n-1}}^C\circ\ldots\circ \varphi_{P_1}^C=\operatorname{id}
$$
and by the induction hypothesis
\begin{equation}\label{eq-II}
\varphi_{P}^D\circ \varphi_{P_{2n-1}}^D\circ\ldots\circ \varphi_{P_1}^D=\operatorname{id}.
\end{equation}
If we replace $\varphi_{P}^D$ in~(\ref{eq-II}) by~(\ref{eq-I}) the claim follows.
\end{proof}
In the case when $\ell$ is tangent to $C$ in the point $R$, Izmestiev gave the closing criterion
\begin{equation}\label{eq-iz}
\frac1{R-P_1}- \frac1{R-P_4}=\frac1{R-P_2}- \frac1{R-P_3}
\end{equation}
for which it is not directly clear that it is projectively invariant. See Figure~\ref{fig-tan}.
\begin{figure}
\begin{center}
\includegraphics{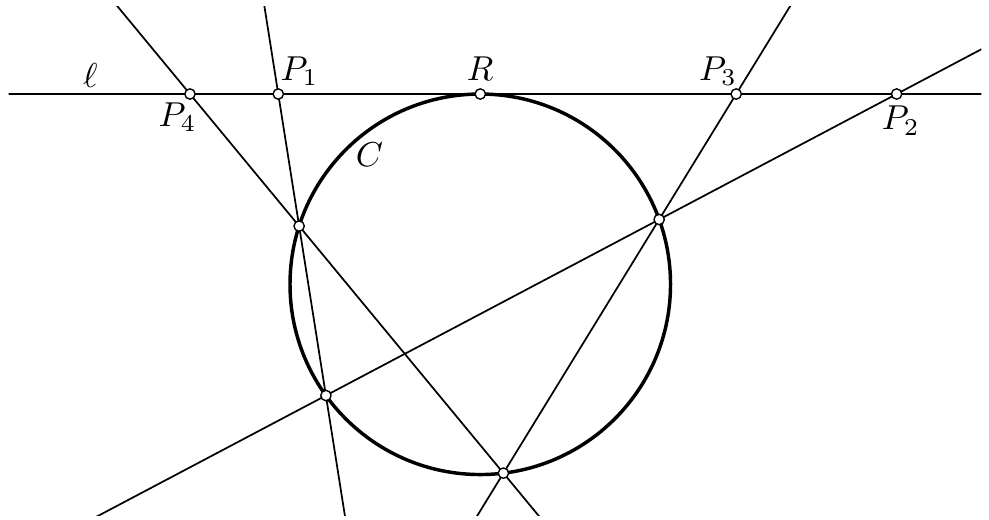}
\caption{The case when $\ell$ is tangent to $C$.}\label{fig-tan}
\end{center}
\end{figure}
The point is, that~(\ref{eq-iz}) can easily be reformulated as
$$
cr(R,P_3;P_1,P_4)=cr(R,P_1;P_3,P_2)
$$
which is obviously a projectively invariant condition.
With the same reasoning as above, we get:
\begin{proposition}
If the points $P_1,\ldots,P_{2n}$ on the line $\ell$ have the closing property with respect
to a conic $C$ and if $\ell$  is tangent to $C$  in the point $R$, then
$P_1,\ldots,P_{2n}$ have the closing property with respect to any other
conic $D$ which is tangent to $\ell$ in~$R$.
\end{proposition}

In the case, when $\ell$ does not meet $C$ we have the following:
\begin{theorem}
If the points $P_1,\ldots,P_{2n}$ on the line $\ell$ have the closing property with respect
to a conic $C$ and if $\ell$  does not meet $C$, then
$P_1,\ldots,P_{2n}$ have the closing property with respect to any other
conic $D$ from the three-dimensional bundle of conics which have the same complex intersections with $\ell$ as $C$.
\end{theorem}
\begin{proof}
As before it suffices to show the claim for $n=2$ to start a proof by induction.
We may again assume that $C$ is given by $\langle X,X\rangle=0$ and the line $\ell$
is given by $\langle X,L\rangle=0$ with $L=(0,0,1)^\top$. Suppose that the points
$P_1,\ldots,P_4$ on $\ell$ satisfy the closing condition with respect to $C$.
It is clear that the points $P_1,\ldots,P_4$ satisfy the closing condition with respect to each member of the three-dimensional
bundle $\mathcal B$ of conics 
$$
\Bigl\langle\begin{pmatrix} x_1-a_1x_3\\x_2-a_2x_3\\r x_3\end{pmatrix}, \begin{pmatrix} x_1-a_1x_3\\x_2-a_2x_3\\r x_3\end{pmatrix}\Bigr\rangle=0.
$$
Each of these conics  intersects $\ell$ in the same points $(1,\pm i,0)$ as $C$.
On the other hand, the bundle $\mathcal C$ of conics through $(1,\pm i,0)$ is precisely
 the bundle $\mathcal B$.
\end{proof}
In particular, for the case of a circle $C$ we obtain:
\begin{corollary}
Let $C$ and $D$ be circles and $\ell$ their radical axis. 
If the points $P_1,\ldots,P_{2n}$ on the line $\ell$ have the closing property with respect
to $C$, then they also have the closing property with respect to $D$ (see Figure~\ref{fig-radical}).
\end{corollary}
\begin{figure}[ht!]
\begin{center}
\includegraphics{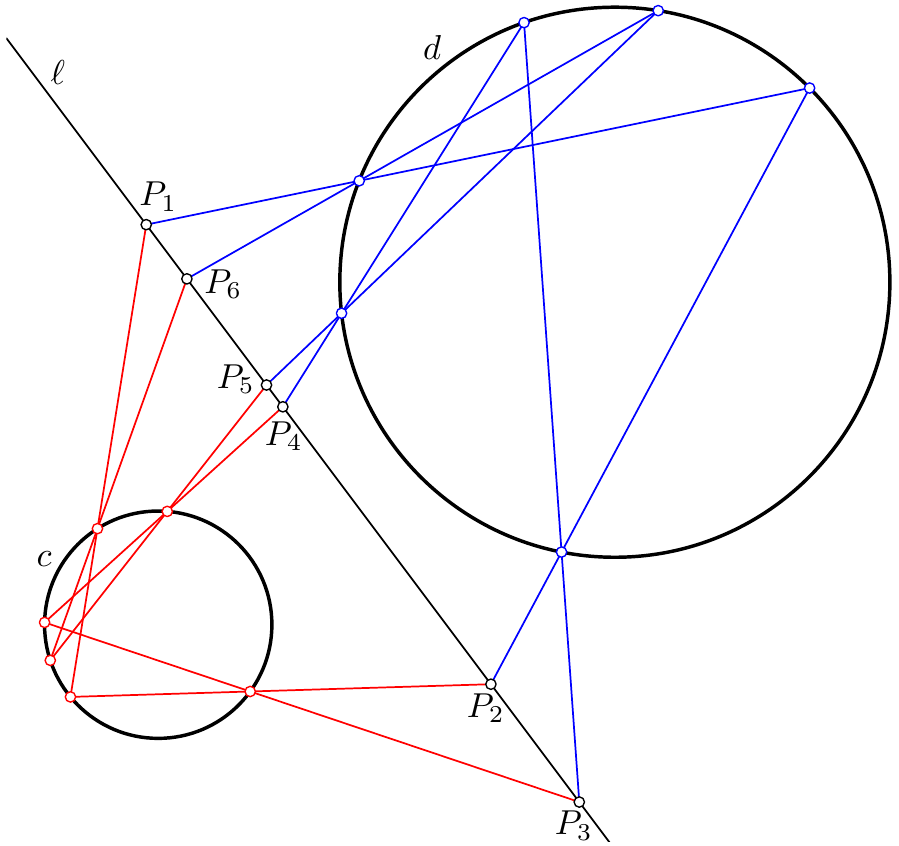}
\caption{Closing property for points on the radical axis of two circles.}\label{fig-radical}
\end{center}
\end{figure}

We end this discussion with the following closing remark: The porisms 
which we presented are  all formulated in the framework of the real projective plane.
This allows to state the dual version in each case. For example, the dual
version of the Butterfly Porism~\ref{thm-kocik} reads as follows:
\begin{theorem}
Let $a_1,a_2,a_3,a_4$ be a quadrilateral circumscribed to a conic $C$ with sides $a_i$, and let $L$ be
a point in $\mathbb R\!\operatorname{P}^2$. Let $p_i$ be the line through $L$ and the intersection of the
sides $a_i$ and $a_{i+1}$ (with cyclically read indices). Then there is a 
quadrilateral circumscribed to $C$ with sides $a'_1,a'_2,a'_3,a'_4$ such that $a'_i$ and $a'_{i+1}$ meet on $p_i$
for any tangent $a'_1$ of $C$ (see Figure~\ref{fig-dualbutter}).
\end{theorem}
\begin{figure}[ht!]
\begin{center}
\includegraphics{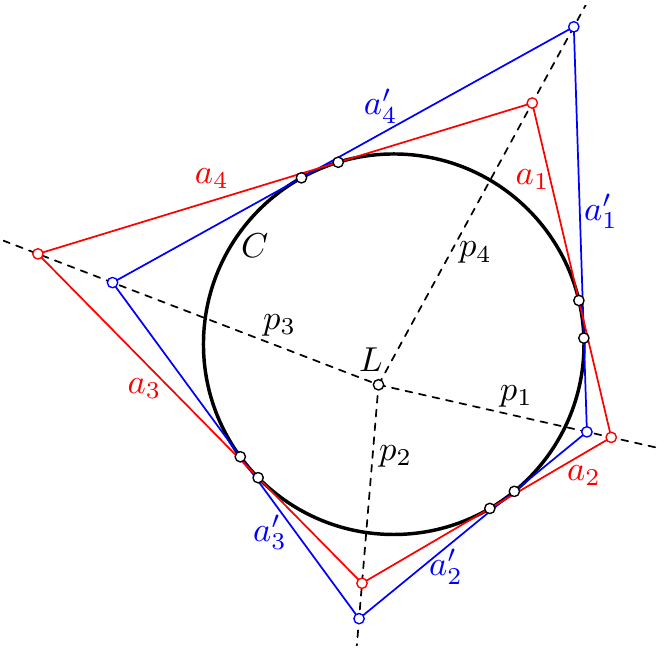}
\caption{The dual Butterfly Theorem.}\label{fig-dualbutter}
\end{center}
\end{figure}

\bibliographystyle{plain}\vspace*{-3mm}

\end{document}